\documentclass[11pt]{amsart} %ArXiv version
\input{header.sty} % headers

\begin{document}

\title{On Quasi-$F$-Purity of Excellent Rings}
\author{Vignesh Jagathese}

\address{Department of Mathematics, Statistics, and Computer Science, University of Illinois at Chicago, Chicago, IL, USA}
\email{vjagat2@uic.edu}
\keywords{Quasi-F-Splittings, Quasi-F-Purity, Algebraic Geometry, Commutative Algebra, Positive Characteristic, Excellent Rings}

\maketitle

\begin{abstract}
 We introduce an analogue to  Quasi-$F$-splittings, Quasi-$F$-purity, which is definable over rings that are not necessarily $F$-finite. We show that this property is equivalent to being Quasi-$F$-split in the complete local and $F$-finite case. We then exhibit that it is stable under completion, direct limit, and local/finite \etale extension. 
\end{abstract}

\section{Introduction}

Introduced in \cite{Yobuko_QuasiFrobeniusSplittingandliftingofCYVarsInCharp}, the quasi-$F$-split height of a characteristic $p$ variety $X/k$ has proven to be a useful and powerful arithmetic and geometric invariant. Though a significantly weaker condition than being Frobenius-split, quasi-$F$-split varieties have been shown to inherit some of the desirable properties of $F$-split varieties, namely $W_2(k)$-liftability \cite{Yobuko_OnTheFrobeniusSplittingHeightofVarsinPositiveChar}, Kodaira Vanishing, and Decomposition of Hodge De Rham \cite{Petrov_DecompositionofDeRhamComplexForQuasiFSplitVarieties}. The Quasi-$F$-split height recovers numerous arithmetic invariants such as the Artin-Mazur height for Calabi-Yau varieties \cite{Yobuko_OnTheFrobeniusSplittingHeightofVarsinPositiveChar}, the $a$-number and $b$-number in specific cases \cite{VanDerGeerKatsura_RelationsBetweenSomeInvariantsofAlgVarsinPosChar}, as well as the order of vanishing of the Hasse Invariant \cite{VanDerGeerKatsura_Aninvariantforvarietiesinposchar}. Just as classical $F$-splittings can be easily detected via Fedder's Criterion \cite{FeddersCrit}, in the case of complete intersections over $F$-finite fields authors in \cite{kawakami2022fedder} (cf. \cite{yoshikawa2025feddertypecriterionquasifesplittingquasifregularity}) have derived a Fedders-style-criterion for computing the quasi-$F$-split height. In an ongoing series of papers \cite{Kawakami++_2022quasifsplittingsbirationalgeometry}, \cite{Kawakami++_2022quasifsplittingsbirationalgeometryII}. and \cite{Kawakami++_2022quasifsplittingsbirationalgeometryIII}, the Quasi-$F$-split height has proven vital in the study of the birational geometry of varieties over perfect fields of positive characteristic, and has also led to more refined invariants such as Quasi-$F^e$-splitting, (uniform) Quasi-$F^\infty$ splitting, Quasi-$F$-regularity and Quasi-$+$-regularity in \cite{tanaka2024quasifesplittings}. \\

Though much has been achieved in the study of quasi-$F$-splittings, less is known when working over varieties over a field $k$ that is not perfect, and even less when $k$ need not be $F$-finite. As seen in lemma $\ref{lem: Witt Noeth iff FFinite}$, being $F$-finite is a sufficient and necessary condition for $W_n(R)$ to be Noetherian, adding an additional and significant layer of difficulty when passing to the non-$F$-finite setting. Just as $F$-purity is a more natural notion than $F$-splitting in the non-$F$-finite setting, we define quasi-$F$-purity and show that it is equivalent to quasi-$F$-splitting over complete local and $F$-finite rings (theorem \ref{Thm: QFSplitIffQFPure}). This is an analogous relationship to $F$-splitting and $F$-purity. \\

We recall that a ring $R$ is a \newword{G-ring} (or Grothendieck ring) if $R_P \to \wh{R_P}$ is geometrically regular for every prime $P$. This is one of requirements for $R$ to be an excellent ring. We prove the following facts about the quasi-$F$-pure height of a local G-ring $R$, denoted $\height(R)$ (Definition \ref{def: QFPurity}).

\begin{theorem*}
Let $\loc$ be a local Noetherian G-ring of characteristic $p$. 
\begin{enumerate}[label = (\alph*)]
    \item  $\height(R) = \height(\wh{R})$. [\ref{cor: HeightCompletes}]
    \item If $R = \dirlim R_i$ is a direct limit of local rings where $\height(R_i) \leq n$ for all $i$ sufficiently large, then $\height(R) \leq n$ [\ref{Thm: DirectLimit}], with equality in special cases [\ref{rem: EqualityifPureLimit}]. 
    \item For any local \etale extension $\loc \to (S,\fn)$, $\height(R) = \height(S)$. [\ref{Thm: LocalEtale}]
    \item For any finite \etale extension $\loc \to S$, $\height(R) = \height(S)$.[\ref{Thm: FiniteEtale}]
    
\end{enumerate}
\end{theorem*}

In particular, $(b)$ and $(d)$ imply that quasi-$F$-pure height is stable under arbitrary separable base change. The author stresses that the separability condition is vital. A counterexample \ref{ex: InsepExtnFails} shows that an inseparable base change of a $2$-quasi-$F$-pure variety need not be quasi-$F$-pure at all; this contrasts the classical notion of $F$-purity, which is stable under arbitrary base change \cite{FeddersCrit}. As an $F$-finite extension of a non-$F$-finite field is necessarily inseparable, one cannot easily base change to the $F$-finite setting. Nevertheless, a more measured approach via the $\Gamma$-construction introduced in \cite{HochsterHuneke94} may be a possible path forward in this regard. A conjecture in this direction will be discussed at the end of the paper. 

\subsection*{Acknowledgements}
The author would like to thank Shiji Lyu, Shravan Patankar, and Kevin Tucker for useful discussions and comments, and the referee for insightful suggestions and corrections. The author was supported in part through NSF RTG Grant DMS-2037569.

\section{Preliminaries}
\label{Witt Intro}
In this section we provide a brief overview of the Witt Vector construction with a view towards the non-$F$-finite setting. In particular, we show in lemma \ref{lem: Witt Noeth iff FFinite} that $W_n(R)$ is not Noetherian; this implies that when $W_n(R)$ is local, it may not admit a faithfully flat completion map. The latter part of this section is dedicated to proving a weaker statement, that $W_n(R) \to W_n(R)\complete$ is pure. In the next section, this fact is used to show that the quasi-$F$-pure property completes.  
\subsection*{Notation}
\begin{itemize}
    \item Unless otherwise stated, All rings $R$ and $S$ will always be Noetherian with identity and of characteristic $p > 0$. 
    \item When $R$ is local, $\fm$ will denote the maximal ideal, $\fK$ the residue field, and $E = E_R(\fK)$ the injective hull of the residue field. 
    \item $W(R)$ is the ring of $p$-typical Witt Vectors with coefficients in $R$, $W_n(R)$ is the ring of truncated Witt Vectors, and $\ov{W_n}(R)$ is the mod $p$ reduction of $W_n(R)$. See the following subsection for definitions of these constructions.
    \item $F: R \to F_*R$ is the Frobenius map $r \mapsto r^p$, where $F_*R$ is the restriction of scalars under Frobenius. A similar convention will be taken for the Witt Frobenius map on $W_n(R)$ and $\ov{W_n}(R)$.
\end{itemize}
\subsection{Witt Vectors}
 For the convenience of the reader we will include a brief review of the Witt Vector construction, though the author recommends \cite{Borger}, \cite[Appendix]{LangerZink}, \cite{DavisKedlaya}, and \cite{LenstraWittVectors} for a more thorough treatment. \\

 For any ring $R$ we define $W(R):= \{(r_0,r_1, \dots) \ | \ r_i \in R\}$ to be the ring of \newword{($p$-typical) Witt Vectors}. We will typically denote elements of $W(R)$ as $\alpha = (\alpha_0,\alpha_1, \dots)$. For any $r \in R$ the element $[r] := (r,0,\dots) \in W(R)$ denotes the \newword{lift} of $r$. Attached to $W(R)$ are addition and multiplication operations defined via universal Witt Polynomials $S_\bullet(\alpha,\beta)$ and $P_\bullet(\alpha,\beta)$ such that
 $$\alpha + \beta = (S_0(\alpha_0,\beta_0), S_1(\alpha_0,\alpha_1,\beta_0,\beta_1), \dots, S_n(\alpha_{\leq n}, \beta_{\leq n}),\dots )$$
 $$\alpha \cdot \beta = (P_0(\alpha_0,\beta_0), P_1(\alpha_0,\alpha_1,\beta_0,\beta_1), \dots, P_n(\alpha_{\leq n}, \beta_{\leq n}),\dots )$$
The construction of these Witt polynomials, at first glance, seems fairly nonstandard. Consider $S_i$ and $P_i$ for $i = 0$ and $1$:
$$S_0(\alpha_0,\beta_0) = \alpha_0 + \beta_0, \qquad S_1(\alpha_0,\alpha_1,\beta_0,\beta_1) = \alpha_1 + \beta_1 - \sum_{i=1}^{p-1}\frac{1}{p}\binom{p}{i}\alpha_0^i\beta_0^{p-i}$$
$$P_0(\alpha_0,\beta_0) = \alpha_0\beta_0, \qquad P_1(\alpha_0,\alpha_1,\beta_0,\beta_1) = \alpha_1\beta_0^p + \alpha_0^p\beta_1 + p\alpha_1\beta_1$$
As $\binom{p}{i}$ is $p$-divisible for all valid $i$, division by $p$ in formula of $S_1$ is purely formal. This paper will focus on Witt Vectors over rings of characteristic $p$; in this case the term $p\alpha_1\beta_1 = 0$ in $P_1(\alpha_{\leq 1}, \beta_{\leq 1})$ and similarly for the other $P_{i}$. We also note that $W(-)$ is functorial $-$ for any ring morphism $f: R \to S$ one can define a ring morphism $W(f): W(R) \to W(S)$ assigning $\alpha \mapsto (f(\alpha_0),f(\alpha_1), \dots)$. The ring of Witt Vectors over a ring of characteristic $p$ has the following associated maps:
\begin{itemize}
    \item Frobenius: $F: W(R) \to W(R)$ is simply the image of $F$ under the functor $W(-)$ assigning $F(\alpha) = (\alpha_0^p,\alpha_1^p,\dots)$. The author stresses that the Witt Frobenius is not the same as the 'standard' Frobenius $\alpha \mapsto \alpha^p$, which is not a ring map on $W(R)$.
    \item Verschiebung: $V: W(R) \to W(R)$ assigning $(\alpha_0,\alpha_1, \dots) \mapsto (0,\alpha_0,\alpha_1, \dots)$. 
\end{itemize}
We define $W_n(R) := W(R)/\im(V^n)$ to be the ring of \newword{$n$-truncated $p$-typical Witt Vectors}, with similar additional and multiplication operations defined by Witt polynomials $S_i,P_i: W_i(R) \times W_i(R) \to R$ for $i < n$. Based on the definition of $S_0$ and $P_0$, it is clear that $W_1(R) \isom R$. This construction is naturally also functorial, and has maps $F: W_n(R) \to W_n(R)$ and $V: W_{n}(R) \to W_{n+1}(R)$ defined similarly to the above, with the additional restriction map $\fR: W_{n+1}(R) \to W_{n}(R)$ assigning $(\alpha_0, \dots, \alpha_n) \mapsto (\alpha_0, \dots, \alpha_{n-1})$. It's worth noting that $F$ and $\fR$ are ring homomorphisms while $V$ is additive but not multiplicative. \\

We note $p \in W_n(R)$ is of the form $(0,1,0, \dots 0)$ and the map $p: \alpha \mapsto p \cdot \alpha$ yields the identity $p = FV = VF$. Let $\ov{W_n}(R) := W_n(R)/p$ be the mod $p$ reduction of $W_n(R)$; when $R$ is perfect (i.e. $F$ is an isomorphism on $R$) one sees that $\im(V) \isom \im(p: \alpha \mapsto p\cdot \alpha)$. This implies that $\ov{W_n}(R) \isom R$, though this isomorphism does not hold in general.  \\

$F_*\ov{W_n}(R)$ is an $R$-module via the action $r \cdot F_*\ov{\alpha} := F_*\ov{\left([r^p] \cdot \alpha\right)}$, and is a $W_n(R)$-Module in the expected way. $R$ is also a $W_n(R)$ module via the restriction map $W_n(R) \xto{\fR^{n-1}} R$. While there certainly can be others, we will assume this $W_n(R)$-module structure on $R$ unless otherwise stated. \\

In the characteristic $p$ setting, Witt Vectors are primarily studied in the case where the underlying ring $R$ is $F$-finite. This is in part due to the following hurdle:
\begin{lemma}
    \label{lem: Witt Noeth iff FFinite}
   For $n > 1$, $W_n(R)$ is Noetherian if and only if $R$ is $F$-finite.
\end{lemma}
\begin{proof}
    It is well known that when $R$ is $F$-finite,  $W_n(R)$ is Noetherian \cite[Proposition A.4]{LangerZink}. For the converse, consider the following short exact sequence:
    $$0 \to \ker(\fR) \to W_n(R) \xto{\fR} W_{n-1}(R) \to 0$$
    Where $\fR$ is the restriction map on Witt Vectors. $\ker(\fR) = \{(0, \dots, 0, r) \in W_n(R) \ | \ r \in R\}$ with a $W_n(R)$-module action $\alpha \cdot V^{n-1}([r]) = V^{n-1}\left(\left[\alpha_{0}^{p^{n-1}}r\right]\right)$. When $R$ is not $F$-finite, $R$ is not a finitely generated $R^{p^{n-1}}$-Module for $n > 1$. It follows then that the ideal $\ker(\fR) \subset W_n(R)$ is not finitely generated and hence, $W_n(R)$ is not Noetherian. 
\end{proof}

\subsection{Purity of Completion of $W_n(R)$}

First we collect a number of facts about $W_n(R)$ from literature.
 \begin{enumerate}[label = (\alph*)]
 \label{lem: facts}
        \item If $R$ is an $F$-finite Noetherian ring, then $W_n(R)$ is a Noetherian ring with finite Witt Frobenius map $F: W_n(R) \to F_*W_n(R)$. \cite[Lemma 2.5]{kawakami2022fedder}.
        \item  If $\loc$ is a local ring, then $W_n(R)$ is a local ring. We reference \cite[Proposition 2.6]{kawakami2022fedder} though the authors' proof also holds in the non-$F$-finite setting. We let $\fm_{W_n(R)}$ denote the maximal ideal of $W_n(R)$ in this case, and provide an explicit presentation of this ideal in the following lemma. 
        \item  If $R \to S$ is an \etale extension of Noetherian rings, then $W_n(R) \to W_n(S)$ is an \etale extension of rings. This was first shown in \cite{KallenWittVectors}, though a more general result can be found in \cite{Borger}.
    \end{enumerate}

For a local ring $\loc$, it makes sense to consider the $\fm_{W_n(R)}$-adic completion of $W_n(R)$. Outside of the $F$-finite setting, lemma \ref{lem: Witt Noeth iff FFinite} ensures that $W_n(R)$ is not Noetherian. Completions of non-Noetherian rings are generally poorly behaved \cite[05JF]{stacks-project}, though fortunately the following lemma ensures that $W_n(R)\complete$ has a well-defined structure.

\begin{lemma}
    Let $\loc$ be a local ring. Then $W_n(\wh{R}) \isom W_n(R)\complete$. In other words, The $\fm_{W_n(R)}$-adic completion of $W_n(R)$ is ring-isomorphic to the image of the $\fm$-adic completion of $R$ under $W_n(-)$.  
    \label{lem: EquivCompletions}
\end{lemma}
\begin{proof}
     Consider the ring morphism $\varphi: W_n(R) \xto{\fR^{n-1}} R \onto R/\fm$. This is surjective, with kernel $J := \left\{(a_0, \dots, a_{n-1}) \in W_n(R) \ | \ a_0 \in \fm\right\}$. As $J$ is the kernel of a surjection onto a field, it is a maximal ideal of $W_n(R)$. As $R$ is local, $W_n(R)$ is local and $J =\fm_{W_n(R)}$. Now consider the ideal
    $$W_n(\fm^k) := \ker(W_n(R) \onto W_n(R/\fm^k)) = \{(a_0, \dots, a_{n-1}) \in W_n(R) \ | \ a_i \in \fm^k \ \forall i\}$$
    One can see that $W_n(\wh{R}) \isom \invlim \frac{W_n(R)}{W_n(\fm^k)}$, so it is sufficient to show that the following two directed systems of $W_n(R)$-modules are cofinal:
    $$\{W_n(\fm^k))\}_k \qquad, \qquad \{J^k\}_k$$
   Fix $k$. We want to find an $\ell$ such that $J^\ell \subseteq W_n(\fm^k)$. Notice that
    $$J^2 \subseteq \{(a_0, \dots, a_{n-1}) \in W_n(R) \ | \ a_0 \in \fm^2, a_1 \in \fm\}$$
    And further
    $$J^t \subseteq \{(a_0, \dots, a_{n-1}) \in W_n(R) \ | \ a_0 \in \fm^t, a_1 \in \fm^{t-1}, \dots, a_{n-1} \in \fm^{t- n + 1}\}$$
    Thus to make sure every term $a_i \in \fm^k$, we just need to $\ell$ such that $k = \ell - n + 1$. Thus $\ell = k + n - 1$ works. Conversely, for a fixed $k$ we want to find an $\ell$ such that $W_n(\fm^\ell) \subseteq J^k$. By similar logic, all $(a_0, \dots, a_{n-1}) \in W_n(\fm^\ell)$  satisfy $a_i \in \fm^\ell$, and thus are contained in $J^k$ for sufficiently large $k$ dependent only on $\ell$ and $n$. The result follows.
\end{proof}
The proof of quasi-$F$-split height completing in the $F$-finite case \cite[Proposition 2.17]{kawakami2022fedder} heavily relies on the fact that $W_n(R)$ is Noetherian, and hence, $W_n(R) \to W_n(R)\complete$ is faithfully flat. In the non-$F$-finite setting such an argument fails. We instead prove a much weaker statement: that $W_n(R) \to W_n(R)\complete$ is a pure map of $W_n(R)$-Modules. We then use that to show $\height(R) = \height(\wh{R})$. \\

We recall for the reader that a map of $R$-modules $M \to N$ is \label{def: pure}\newword{pure} if for any $R$-Module $L$, $L \tensor_{R} M \to L \tensor_{R} N$ is injective. it is clear that split maps and faithfully flat maps are pure; in addition when $\locr$ is a Noetherian local ring, $M \to N$ is pure if and only if $E \tensor_R M \to E \tensor_R N$ is injective. As injectivity is a local property, so is purity. Finally, we say that $R$ is $F$-pure if $F:R \to F_*R$ is a pure map of $R$-modules. \\

Before we proceed with verifying $W_n(R) \to W_n(\wh{R})$ is pure, however, we will need a slightly weaker though nonetheless interesting result:

\begin{proposition} 
\label{lem: WnSmoothToPure} 
    $W_n(-)$ takes smooth extensions of Noetherian rings to pure extensions of rings. 
\end{proposition}
\begin{proof}

    Both smoothness and purity are local properties, so we can assume our extensions are local extensions. Let $\varphi: R \to S$ be a smooth extension of Noetherian local rings. Via \cite[054L]{stacks-project} we see that we can factor this morphism through $R[x_1, \dots, x_d]$ for $d$ the relative dimension of $S$ over $R$:
    \begin{center}
        % https://tikzcd.yichuanshen.de/#N4Igdg9gJgpgziAXAbVABwnAlgFyxMJZABgBpiBdUkANwEMAbAVxiRACUQBfU9TXfIRQBGclVqMWbdsgAeAfVEACADoqoEHHFJKFUCt14gM2PASKjh4+s1aIQAZW7iYUAObwioAGYAnCAC2SGQgOBBIwjw+-kGIoqHhiABM1DZS9mpoWIbRgcHUYUgpErZsavS+aAAW2VwUXEA
\begin{tikzcd}
R \arrow[r] \arrow[rd, "\varphi"'] & {R[x_1, \dots, x_d]} \arrow[d, "\pi"] \\
                                  & S                                    
\end{tikzcd}
    \end{center}
    Where $\pi$ is \etale. Passing this diagram through $W_n(-)$ yields
    \begin{center}
        % https://tikzcd.yichuanshen.de/#N4Igdg9gJgpgziAXAbVABwnAlgFyxMJZABgBpiBdUkANwEMAbAVxiRAHUB9MACgCUAlCAC+pdJlz5CKAIzkqtRizZdefZAA9OcgAQAdPVAg44pHVqgUho8djwEicmQvrNWiDtx4BlawphQAObwRKAAZgBOEAC2SGQgOBBIMjYgkTHJ1IlIAEzUrsoeqjwGaFjWYmlRsYjx2Yh5im4qXgb0EWgAFuUiFMJAA
\begin{tikzcd}
W_n(R) \arrow[r] \arrow[rd, "W_n(\varphi)"'] & {W_n(R[x_1, \dots, x_d])} \arrow[d, "W_n(\pi)"] \\
                                            & W_n(S)                                         
\end{tikzcd}
    \end{center}
    $R \to R[x_1, \dots, x_d]$ splits in the category of rings, so via functionality of $W_n(-)$ we can conclude that $W_n(R) \to W_n(R[x_1, \dots, x_d])$ also splits as a ring homomorphism. Thus, it splits as a morphism of $W_n(R)$-Modules, and is hence pure. $W_n(\pi)$ is \etale and hence is flat. As $W_n(\pi)$ is a local morphism of local rings, this is equivalent to faithful flatness, and all faithfully flat maps are pure. It follows that $W_n(\varphi)$ is a composition of pure maps, and hence it is pure.  
\end{proof}

The author stresses that purity is a much weaker condition than flatness. Indeed, $W_n(\varphi)$ as above is not flat outside of the case where $\varphi$ is \etale.

\begin{remark}
     When $d > 0$ as above (i.e. $\varphi$ is smooth but not \etale ) then $W_n(\varphi)$ is pure but necessarily not flat (nor even faithful). 
\end{remark}
\begin{proof}
    $W_n(R) \to W_n(R[x_1, \dots, x_d])$ is not flat for any choice of $n > 1, d > 0$, and thus 
    $$\Tor_{W_n(R)}^1(W_n(R[x_1, \dots, x_d]), M) \neq 0$$
    for some $W_n(R)$-module $M$. As $W_n(\pi)$ is \etale, it is pure, so by \cite{HochsterRoberts} we have an embedding of Tors 
$$0 \neq \Tor_{W_n(R)}^1(W_n(R[x_1, \dots, x_d]), M) \into \Tor_{W_n(R)}^1(W_n(S), M)$$
implying that $W_n(S)$ is not a flat $W_n(R)$ module. 
\end{proof}

It can be readily seen that any smooth extension $R \to S$ is geometrically regular; even a direct limit of smooth extensions $\dirlim S_i$ is geometrically regular. In \cite{NeronPopescu},Popescu proved the remarkable fact that the converse also holds; i.e. any geometrically regular extension of Noetherian rings $R \to S$ can be realized as a filtered colimit of smooth extensions $R \to S_i$ where $\dirlim S_i = S$. We will cite this to extend proposition \ref{lem: WnSmoothToPure} to geometrically regular extensions, but first we prove an intermediate lemma:
\begin{lemma}
    \label{lem: WittVectorCommutesDirectLimit}
    Let $R = \dirlim R_i$ be a direct limit in the category of rings. Then $W_n(R) \isom \dirlim W_n(R_i)$ as rings, i.e. $W_n(-)$ commutes with direct limits. 
\end{lemma}
\begin{proof}
    As $W_n(R_i)$ consists of finite length vectors with coefficients in $R_i$, it can be easily deduced that $W_n(\dirlim R_i)$ satisfies the universal property of colimits with respect to the directed system $\{W_n(R_\bullet)\}$.
\end{proof}
Given the Witt Vector construction commutes with colimits, we can proceed with the theorem. 
\begin{theorem}
     $W_n(-)$ takes geometrically regular extensions of Noetherian rings to pure extensions of rings. 
\end{theorem}

\begin{proof}
    Suppose $R \to S$ is a geometrically regular extension as above. Popescu's theorem tells us that $S$ is a filtered colimit of smooth $R$-algebras $S = \dirlim S_i$. $W_n(R) \to W_n(S_i)$ for each $i$ is hence a pure morphism by proposition \ref{lem: WnSmoothToPure}, and as purity is preserved under colimits, it follows that $W_n(R) \to \dirlim W_n(S_i)$ is pure. via lemma \ref{lem: WittVectorCommutesDirectLimit} it follows that $ \dirlim W_n(S_i) =   W_n(\dirlim S_i) = W_n(S)$. Thus we conclude that $W_n(R) \to W_n(S)$ is pure. 
\end{proof}
\begin{corollary}
\label{cor: GRingImpliesPureWitt}
    If $\loc$ is a local G-ring, then $W_n(R) \to W_n(\wh{R}) = W_n(R)\complete$ is pure. 
\end{corollary}

\begin{proof}
    This immediately follows from the definition of a G-ring (see \cite[07GG]{stacks-project}). In particular, this result holds when $R$ is excellent. 
\end{proof}

The author is unaware whether or not $W_n(R) \to W_n(\wh{R}) = W_n(R)\complete$ is flat when $R$ is not $F$-finite. Regardless, following our definition of Quasi-$F$-purity we will use this corollary to conclude that the Quasi-$F$-pure height is preserved after completion.

\section{Quasi-$F$-Purity}
A ring $R$ of positive characteristic is \newword{$n$-Quasi-$F$-Split} \cite{Yobuko_QuasiFrobeniusSplittingandliftingofCYVarsInCharp} if there exists a $W_n(R)$-Module homomorphism $F_*W_n(R) \to R$ that makes the following diagram commute:
\begin{center}
% https://tikzcd.yichuanshen.de/#N4Igdg9gJgpgziAXAbVABwnAlgFyxMJZABgBpiBdUkANwEMAbAVxiRAHUB9MACgB0+AeQAeAShABfUuky58hFAEZyVWoxZsAYpwBUXXgJHipM7HgJEyi1fWatEIQ8MmqYUAObwioAGYAnCABbJDIQHAgkACZqWw0HAR8AJQA9YDAAWkUJEGoGOgAjGAYABVlzBRA-LHcACxxJaRB-IJDqcKRlNTstBt8A4MRO9sRokAYsMHsQKDo4GrcXCSA
\begin{tikzcd}
W_n(R) \arrow[d, "\fR^{n-1}"'] \arrow[r, "F"] & F_*W_n(R) \arrow[ld, dashed] \\
R                                             &                               
\end{tikzcd}
\end{center}
Letting $Q_{R,n}$ be the pushout of the diagram above, one can check that the existence of a quasi-$F$-splitting is equivalent to the existence of a splitting of the map $\Phi_{R,n}$ below:
\begin{center}
% https://tikzcd.yichuanshen.de/#N4Igdg9gJgpgziAXAbVABwnAlgFyxMJZABgBpiBdUkANwEMAbAVxiRAHUB9MACgB0+AeQAeAShABfUuky58hFAEZyVWoxZsAYpwBUXXgJHipM7HgJEyi1fWatEIQ8MnSQGM-KLLr1WxocAipzAABqkYBKSqjBQAObwRKAAZgBOEAC2SGQgOBBIAEy+6vaOfEkASgB6wGAAtIqR1Ax0AEYwDAAKsuYKIClYsQAWOC7JaZmIyjl5iADMJiCpGQXUuUizRXZsAh2DWMFhESBNre1dHhYO-UMjC0sTG9MrIG1gUEi1s9kMWGAlUHQ4IMYqNFuMsqsZlM-CVNFEJEA
\begin{tikzcd}
W_n(R) \arrow[d, "\fR^{n-1}"'] \arrow[r, "F"] & F_*W_n(R) \arrow[d]                   \\
R \arrow[r, "{\Phi_{R,n}}"']                  & {Q_{R,n}} \arrow[l, dashed, bend right]
\end{tikzcd}
\end{center}
Following the pushout construction, we can identify $Q_{R,n} = F_* \ov{W_n}(R)$, and that $\Phi_{R,n}$ is defined by the assignment $r \mapsto F_*(\ov{[r^p]})$. 

\begin{remark}
    $Q_{R,n}$ is naturally an $R$-Module via the action $r \cdot F_*\ov{\alpha} = F_*\ov{\left([r^p] \cdot \alpha\right)}$. As completion commutes with both Frobenius pushforward and mod $p$ reduction, it is clear from lemma \ref{lem: EquivCompletions} that $Q_{R,n}\complete \isom Q_{\wh{R},n}$ as $\wh{R}$-Modules.  When $R$ is $F$-finite,  $Q_{R,n}$ is a finite $R$-Module, in which case $\wh{R} \tensor_R Q_{R,n} \isom Q_{\wh{R},n}$ also holds. This construction is functorial; i.e. for any map of rings $R \to S$ and any $n \in \NN$ there is a corresponding map of $R$-Modules $Q_{R,n} \to Q_{S,n}$. 
\end{remark}
\begin{lemma}
    $R$ is reduced if and only if $\Phi_{R,n}$ is an injective $R$-Module homomorphism for some (equivalently, all) $n \in \NN$. 
    \label{lem: reducediffinj}
\end{lemma}
\begin{proof}
    Fix $n \in \NN$. We will prove the equivalent statement that $F: R \to F_*R$ is not injective if and only if $\Phi_{R,n}$ is not injective. We are most of the way there fairly easily:
    $$0 \neq t \in \ker(F) \iff t^p = 0 \iff F_*[t^p] = 0 \Rightarrow F_*\ov{[t^p]} = 0 \iff 0 \neq t \in \ker(\Phi_{R,n})$$
    Recall that $\forall t \in R$ nonzero, $[t] = (t,0, \dots, 0) \not\in (p) = ((0,1,0,\dots,0))$. It follows that $F_*[t^p] = 0 \iff F_*\ov{[t^p]} = 0$ as desired. 
\end{proof}

Testing for Frobenius splitting is is not a useful method to detect singularities outside of the $F$-finite case. Indeed, there exists excellent local Henselian DVRs \cite{DattaMurayama_TateAlgebrasandFrobeniusNon-SplittingofExcellentRegularRings} over non-$F$-finite fields that are not $F$-split. This ring is $F$-pure however, and as faithfully flat maps are pure, it immediately follows from \cite{Kunz_FlatIffRegular} that regular rings are always $F$-pure. Thus, testing for purity is far more natural to do in the non-$F$-finite setting. This motivates the following definition: 

\begin{definition}
\label{def: QFPurity}
    Let $R$ be a Noetherian ring of characteristic $p$. $R$ is \newword{$n$-Quasi-$F$-pure} if $R \to Q_{R,n}$ is a pure $W_n(R)$-Module homomorphism. If such an $n$ exists, $R$ is \newword{Quasi-$F$-Pure}. Let $\height(R)$ denote the \newword{Quasi-$F$-pure height}, or the minimal $n$ such that $R$ is $n$-Quasi-$F$-pure, where $\height(R) = \infty$ if it is not $n$-Quasi-$F$-pure for any $n$. 
\end{definition}
It is easy to see that if $R$ is $n$-Quasi-$F$-pure, then it is also $n+1$-Quasi-$F$-pure, implying our definition of height is well founded. As $F_*\ov{W_1}(R) \isom F_*R$, $R$ is $F$-pure if and only if it is $1$-quasi-$F$-pure. Notably, as pure maps must be injective to begin with, lemma \ref{lem: reducediffinj} above has the following obvious (and useful) corollary:
\begin{corollary}
    If $R$ is Quasi-$F$-Pure then $R$ is reduced.
    \label{cor: QFPureImpliesReduced}
\end{corollary}
The converse is not true; there are various examples in \cite{kawakami2022fedder} of integral $F$-finite schemes that are not Quasi-$F$-split, and hence not Quasi-$F$-pure. Similar to the classical notions of $F$-Splitting and $F$-Purity, these notions are equivalent when $R$ is $F$-finite or is a complete local ring.
\begin{theorem}
    Suppose $R$ is a local ring. If $R$ is either complete OR $F$-finite, then being $n$-Quasi-$F$-Split and $n$-Quasi-$F$-Pure are equivalent notions. In particular, the Quasi-$F$-pure height and the Quasi-$F$-split height of $R$ are the same. \label{Thm: QFSplitIffQFPure}
\end{theorem}
\begin{proof}

    All maps that split are pure, so it is sufficient to check that, when $R$ is complete local or $F$-finite, that being $n$-Quasi-$F$-pure implies $R$ is $n$-Quasi-$F$-Split. \\

    Suppose the map $\Phi_{R,n}: R \to Q_{R,n}$ is a pure map of $R$-Modules. 
    \begin{itemize}
        \item If $R$ is a complete local ring, then a Matlis duality argument \cite[lemma 1.2]{FeddersCrit} implies that $\Phi_{R,n}$ splits.
        \item  Now assume $R$ is $F$-finite. To show $\Phi_{R,n}$ splits, it is equivalent to showing that the evaluation map $\Hom_R(Q_{R,n},R) \to R$ is surjective. $\wh{R}$ is a faithfully flat $R$-module, so one only needs to show that
        $$\wh{R} \tensor_R \Hom_R(Q_{R,n},R) \onto \wh{R}$$
        When $R$ is $F$-finite, $F_*\ov{W_n}(R) = Q_{R,n}$ is a finitely generated $R$-Module. Therefore, 
        $$\wh{R} \tensor_R \Hom_R(Q_{R,n},R) \isom \Hom_{\wh{R}}(Q_{R,n}\complete,\wh{R})\isom \Hom_{\wh{R}}(Q_{\wh{R},n},\wh{R})$$
        And we know $\Hom_{\wh{R}}(Q_{\wh{R},n},\wh{R}) \onto \wh{R}$ from the previous case.
        
    \end{itemize}

\end{proof}

\begin{lemma}
    Suppose $R \to S$ is a pure extension of rings. Then $\height(R) \leq \height(S)$. \label{lem: PureExtn}
\end{lemma}
\begin{proof}
Take the following diagram of $R$-Modules:
    \begin{center}
        % https://tikzcd.yichuanshen.de/#N4Igdg9gJgpgziAXAbVABwnAlgFyxMJZABgBpiBdUkANwEMAbAVxiRACUQBfU9TXfIRQBGclVqMWbAGIB9AFQAdRRBrAA6rLBcAFOwCU3XiAzY8BIqOHj6zVohBylKtZu06AyoZ58zgomTW1LZSDh7c4jBQAObwRKAAZgBOEAC2SGQgOBBIwj4gyWlIAMzU2UgATPmF6YiZ5YjF1Sm1olk5iFUUXEA
\begin{tikzcd}
R \arrow[r] \arrow[d] &Q_{R,n} \arrow[d] \\
S \arrow[r]           & Q_{S,n}          
\end{tikzcd}
    \end{center}
    As $R \to S$ is pure, it follows that if $S \to  Q_{S,n}$ is pure, so is $R \to Q_{S,n}$. Thus the composition $R \to Q_{R,n} \to Q_{S,n}$ is pure, so $R \to Q_{R,n}$ is pure. 
\end{proof}

%Expand on this above? Issue is that S->Qs,n is known to be S-pure, not R-pure. We are assuming that the purity property descends (i.e. if S is an R-alg and M->N is S-pure, it is R-pure 

\subsection*{Stability under Completion}
As $R$ is Noetherian the map $R \to \wh{R}$ is faithfully flat, hence pure. Thus lemma \ref{lem: PureExtn} implies that $\height(R) \leq \height(\wh{R})$. We'd like to know whether $=$ holds; while unknown in general, equality holds given a seemingly minor purity condition. 
\begin{lemma}
    Let $\locr$ be a local ring. If $W_n(R) \to W_n(\wh{R})$ is a pure map for any $n$, then $\height(R) = \height(\wh{R})$. 
\end{lemma}
\begin{proof}
  Given $Q_{R,n} = F_*\ov{W_n}(R)$, one easily sees that purity of $W_n(R) \to W_n(\wh{R})$ implies the purity of $Q_{R,n} \to Q_{\wh{R},n}$. Recall that $E = E_R(\fK) = E_{\wh{R}}(\fK)$. $ Q_{R,n} \to Q_{\wh{R},n}$ is pure, so $E \tensor_R Q_{R,n} \to E \tensor_R Q_{\wh{R},n}$ is injective. From this criterion we have the following chain of equivalences:
    \begin{flalign*}
        R \text{ is $n$-Quasi-$F$-pure} &\iff \Phi_{R,n} \text{ is pure} \iff \ker(E \to E \tensor_R Q_{R,n}) = 0 \\
        &\iff \ker(E \to E \tensor_R Q_{R,n} \into E \tensor_R Q_{\wh{R},n}) = 0 \\
        &\iff \ker(E = E \tensor_{\wh{R}} \wh{R} \to E \tensor_{\wh{R}} Q_{\wh{R},n}) = 0 \\
         &\iff \Phi_{\wh{R},n} \text{ is pure} \iff \wh{R} \text{ is $n$-Quasi-$F$-pure}
    \end{flalign*}
\end{proof}
This purity condition is known to hold for $F$-finite rings, as $W_n(R)$ is Noetherian, and thus has faithfully flat (hence pure) completion. When $R$ is not $F$-finite, $W_n(R)$ is necessarily non-Noetherian, but corollary \ref{cor: GRingImpliesPureWitt} implies that this purity condition holds for excellent rings (or more generally, G-rings). This discussion yields the following corollary:
\begin{corollary}
\label{cor: HeightCompletes}
    If $R$ is a G-Ring (e.g. if $R$ is excellent) then $\height(R) = \height(\wh{R})$. 
\end{corollary}
While being a G-ring is sufficient, it is not known to the author what hypothesis are necessary on $R$ for this equality on height to hold.

\subsection*{Stability under Direct Limit}
It is well known that within the category of $R$-Modules, purity is preserved under direct limit. We prove a slightly more general claim that purity is preserved even when the base ring changes under the limit. This will be necessary for confirming that quasi-$F$-purity is preserved under direct limit.
\begin{lemma}
    Let Let $M = \dirlim M_i$ and $N = \dirlim N_i$ be two directed systems of $\dirlim R_i = R$-Modules where $M_i \to N_i$ is pure as a morphism of $R_i$-Modules for each $i$. Then $M \to N$ is a pure map of $R$-Modules, i.e. purity is preserved under direct limit.
\end{lemma}
\begin{proof}
    Choose any $R$-Module $L$. Letting $R_i \to R \to L$ define an $R_i$-Module structure on $L$, we immediately see that $L \tensor_{R_i} M_i \to L \tensor_{R_i} N_i$ is injective for all $i$ by purity. Further, $\dirlim (L \tensor_{R_i} M_i) \to \dirlim (L \tensor_{R_i} N_i)$ is injective, as direct limits are exact in the category of abelian groups and thus preserve injectivity. Thus this map is an injective morphism of $R$-Modules. Via tensoring the universal property diagram of $\dirlim M_i$ with $L$ (as an $R_i$-Module for each $M_i$ and as an $R$-Module for $M$), we see that $L \tensor_R \dirlim M_i$ satisfies the universal property of direct limits with respect to the directed system $\{L \tensor_{R_i} M_i\}_{i \in \NN}$. It follows that
    $$\dirlim (L \tensor_{R_i} M_i) \isom L \tensor_R \dirlim M_i$$
    As $R$-Modules, and similarly for the $N_i$. Thus
    $$L \tensor_R M \isom L \tensor_R \dirlim M_i \to L \tensor_R \dirlim N_i \isom L \tensor_R N$$
    is an injective map of $R$-Modules, and hence, $M \to N$ is pure. 
\end{proof}
\begin{theorem}
\label{Thm: DirectLimit}
   Let $\{S_i\}$ be a directed system of rings where $\dirlim S_i = S$. If $\exists M \in \NN$ such that $\forall m > M$ $\height(S_m) \leq n$, then $\height(S) \leq n$. 
\end{theorem}
\begin{proof}
    Fix $M$ as above. If $\height(S_m) \leq n$, then $S_m \to Q_{S_m,n}$ is a pure map of $S_m$-Modules for all $m > M$. Via the prior lemma, it follows that $S = \dirlim S_i \to \dirlim Q_{S_i,n}$ is a pure map of $S$-Modules. \\ 
    
    Further, recall that direct limits commute with the Witt Vector construction (lemma \ref{lem: WittVectorCommutesDirectLimit}). This, along with the fact that pushouts are themselves colimits and thus commute with direct limits, implies that $\dirlim Q_{S_i,n} = Q_{S,n}$. It follows that $S \to Q_{S,n}$ is pure. Thus, $\height(S) \leq n$. 
\end{proof}
\begin{remark}
\label{rem: EqualityifPureLimit}
    if $\height(S_i) = n$ and this is a pure direct limit (i.e. $S_i \to S$ is a pure $S_i$-module homomorphism for all $i$ sufficiently large), then via lemma \ref{lem: PureExtn} we get equality, i.e. $\height(S) = n$ for $n$ as above. 
\end{remark}

\subsection*{Stability under \'Etale Extension}

For $S$ an $R$-algebra, it is well known that the relative Frobenius morphism $F_{S/R}: F_*R \tensor_R S \to F_*S$ is an isomorphism when $R \to S$ is \etale, see \cite[0EBS]{stacks-project}. A similar statement holds for the similarly defined relative Witt Frobenius.
\begin{lemma}
\label{lem: RelWittFrobeniusIsomorphism}
    Let $R \to S$ be an \etale extension. Then the relative Witt Frobenius 
    $$W_n(F)_{S/R}: F_*W_n(R) \tensor_{W_n(R)}W_n(S) \to F_*W_n(S)$$
    is an isomorphism. 
\end{lemma}
\begin{proof}
    We reference \cite[Proposition A.12]{LangerZink}. Using their notation, set $R' = F_*R$. Then $S' := F_*R \tensor_R S \isom F_*S$, where the isomorphism follows from the fact that $R \to S$ is \etale. As the pullback along Frobenius commutes with the Witt functor, i.e. $W_n(F_*R) \isom F_*W_n(R)$ and similarly for $S$, we obtain our desired result.
\end{proof}
From this, we can deduce that, for any local \etale extension $R \to S$ of a G-ring $R$, $\height(R) = \height(S)$. 
\begin{theorem}
\label{Thm: LocalEtale}
    Let $\loc$ be a G-ring and $\loc \to (S,\fn)$ be a local \etale extension of local rings. Then $S$ is a G-ring and $\height(R) = \height(S)$.
\end{theorem}
\begin{proof}
    First note that $S$ is essentially of finite type over $R$; it follows from \cite[07PV]{stacks-project} that $S$ is also a G-ring. In this setting via corollary \ref{cor: HeightCompletes}, we know that $\height(R) = \height(\wh{R})$,  $\height(S) = \height(\wh{S})$, and via \cite[039M]{stacks-project} that $\wh{R}_\fm \to \wh{S}_\fn$ remains \etale. Thus, we may reduce to the case where $R$ and $S$ are both complete local, where quasi-$F$-split and quasi-$F$-pure height coincide (Theorem \ref{Thm: QFSplitIffQFPure}). \\

    We first note that \etale maps are (faithfully) flat, hence pure. Thus $\height(R) \leq \height(S)$. As quasi-$F$-purity and quasi-$F$-splitting are equivalent notions over complete local rings, it is sufficient to check that if $R$ is $n$-quasi-$F$-split, then so is $S$. If $R$ is $n$-quasi-$F$-split there exists the following dashed morphism of $W_n(R)$-modules:
  \begin{center}
        % https://tikzcd.yichuanshen.de/#N4Igdg9gJgpgziAXAbVABwnAlgFyxMJZABgBpiBdUkANwEMAbAVxiRAHUB9MACgCUAlCAC+pdJlz5CKAIzkqtRizYAxTgCouvQSLEgM2PASJkZC+s1aIQfEQphQA5vCKgAZgCcIAWyRkQOBBIMqLuXr6I-oFIAEyhIJ4+wdTRiDHUDFhgViBQdHAAFg52wkA
\begin{tikzcd}
W_n(R) \arrow[r] \arrow[d] & F_*W_n(R) \arrow[ld, dashed] \\
R                          &                             
\end{tikzcd}
    \end{center} 
    Applying the functor $-\tensor_{W_n(R)}W_n(S)$ yields the diagram of $W_n(S)$-Modules

    \begin{center}
        % https://tikzcd.yichuanshen.de/#N4Igdg9gJgpgziAXAbVABwnAlgFyxMJZABgBpiBdUkANwEMAbAVxiRAHUB9MACgGUAlCAC+pdJlz5CKAIzkqtRizYAxTgCouvAEoCABAB0DOGGDgQATp2Baeu4bcEixIDNjwEiZGQvrNWiCDahsam5lY23HYCwnqOQsIKMFAA5vBEoABmFhAAtkhkIDgQSDKiWTn5iIXFSABM5SDZeaXUtYh11AxYYAEgUHRwABbJIhTCQA
\begin{tikzcd}
W_n(S) \arrow[r] \arrow[d] & F_*W_n(R) \tensor_{W_n(R)}W_n(S) \arrow[ld, dashed] \\
R \tensor_{W_n(R)} W_n(S)  &                                                    
\end{tikzcd}
    \end{center}
   Via lemma \ref{lem: RelWittFrobeniusIsomorphism}, we can see that $F_*W_n(R) \tensor_{W_n(R)} W_n(S) \isom F_*W_n(S)$. Further, we have a restriction morphism 
    $$\Id \tensor_{W_n(R)}\fR^{n-1}: R \tensor_{W_n(R)}W_n(S) \to R \tensor_{W_n(R)}S$$
    As $R \to S$ is \'etale, this map is an isomorphism \cite[Proposition A.8]{LangerZink}. Utilizing this alongside the $W_n(R)$-module structure of $R$ and $S$ yields the following isomorphism of $W_n(S)$-modules.
    $$ R \tensor_{W_n(R)}W_n(S) \isom  R \tensor_{W_n(R)}S \isom R \tensor_R S \isom S$$ 
   Plugging both of these isomorphisms into the above diagram yields
    \begin{center}
\begin{tikzcd}
W_n(S) \arrow[r] \arrow[d] &F_*W_n(S) \arrow[ld, dashed] \\
S &                                                    
\end{tikzcd}
    \end{center}
Where the dashed morphism denotes the desired $n$-Quasi-$F$-splitting.
\end{proof}

\begin{corollary}
    Suppose $\loc$ is a local G-ring and $(S,\fn)$ is a direct limit of local \etale $R$-algebras. Then $\height(R) = \height(S)$.
\end{corollary}
\begin{proof}
    $R \to S$ is a direct limit of pure extensions, and hence pure. Thus by lemma \ref{lem: PureExtn} $\height(R) \leq \height(S)$. Checking $\height(S) \leq \height(R)$ is a combination of theorems \ref{Thm: DirectLimit} and \ref{Thm: LocalEtale}.
\end{proof}

\begin{corollary}
\label{cor: sepExtension}
    Let $\loc$ be a local G-ring and $k$-algebra, for $k$ any field of characteristic $p > 0$. If $L/k$ is an algebraic separable extension, then $\height(R) = \height(L \tensor_{k} R)$. 
\end{corollary}
\begin{proof}
    $L = \dirlim L_i$ where each $L_i$ is a finite separable, hence \etale, extension of $k$. It follows then $R \to L_i \tensor_{k} R$ is a local \etale extension, as \etale extensions are preserved under base change. Thus, 
    $$L \tensor_{k} R \isom (\dirlim L_i) \tensor_{k} R \isom \dirlim (L_i \tensor_{k} R)$$
    is a direct limit of local \etale $R$-algebras, and we cite the preceeding corollary. 
\end{proof}
As $\height(R) = \height(\wh{R})$ and $\wh{R}$ is naturally an algebra over a coefficient field $\fK$, corollary \ref{cor: sepExtension} shows that $\height(R) = \height(L \tensor_\fK \wh{R})$ for $L$ any algebraic separable extension of $\fK$. \\

We can generalize theorem \ref{Thm: LocalEtale} outside of the local case, but will require that $R \to S$ be module-finite.
\begin{theorem}
\label{Thm: FiniteEtale}
    Let $\loc$ be a local G-ring and $R \to S$ be a finite \etale extension. Then $S$ is a G-ring and $\height(R) = \height(S)$. 
\end{theorem}

\begin{proof}
    First we base change to the completion $\wh{R}$, where the corresponding map $\wh{R} \to \wh{R} \tensor_R S$ remains \etale. Via \cite[04GH]{stacks-project} we can decompose $\wh{R} \tensor_R S \isom S_1 \times \dots \times S_\ell$ where each $S_i$ is local (with maximal ideal denoted $\fn_i$) and finite over $\wh{R}$. From this description, it easy to see that $S_{\fn_i} \isom S_i$ for each $i$, and as $\fn_i$ lies over $\fm$, the composition $\wh{R} \to S \to S_{\fn_i} \isom S_i$ is a local map. As localization is \etale it follows that $\wh{R} \to S_i$ is \etale local. Thus via corollary \ref{cor: HeightCompletes} and theorem \ref{Thm: LocalEtale}, $\height(R) = \height(\wh{R}) = \height(S_i)$ $\forall i \leq \ell$. As purity is a local property that can be checked at the level of max Spec, we see that
    $$\height(S) = \sup_{i \leq \ell}\height(S_i) = \height(R)$$
\end{proof}
\section{Conjecture on The $\Gamma$-Construction}
Unfortunately, separability in corollary \ref{cor: sepExtension} is a required assumption.
\begin{remark}
    Consider an augmentation of \cite[Example 7.13]{kawakami2022fedder}: 
    $$R = \frac{\FF_2(S,T)[[x,y,z]]}{(Sx^2 + Ty^2 + z^2)}$$ 
    A computation in \cite{kawakami2022fedder} shows that this can realized as the general fiber of a quasi-$F$-split height $2$ fibration, and is hence also quasi-$F$-split. A Fedder's criterion computation shows this is not $F$-split, so $2 \leq \height(R) < \infty$.  However, the base change to the algebraic (or even perfect) closure $\ov{\FF_2(S,T)} \tensor_{\FF_2(S,T)} R$ is non-reduced, as $(\sqrt{S} x + \sqrt{T} y + z)^2 = 0$. Thus, by corollary \ref{cor: QFPureImpliesReduced}, $\ov{\FF_2(S,T)} \tensor_{\FF_2(S,T)} R$ is not quasi-$F$-split. 
    \label{ex: InsepExtnFails}
\end{remark}
Thus, base change under inseparable extensions can fail spectacularly, even over $F$-finite irreducible hypersurfaces. This is in contrast to the standard notion of $F$-purity, which is stable under arbitrary base change \cite[proposition 1.11]{FeddersCrit}. You can, however, perform arbitrary base change under specific circumstances, at least over $F$-finite fields. See Theorem 5.13 and Proposition 6.10 in \cite{kawakami2022fedder} for examples. \\

That is not to say that quasi-$F$-pure height is not stable under any inseparable base change. In the example above, the $p$-base of $\FF_2(S,T)$ is $\Lambda := \{1,S,T\}$. If we only adjoin roots for $S$ but not $T$, then the extension $L = \bigcup_{i\in \NN}\FF_2(S^{1/2^i},T)$ is a purely inseparable extension of $\FF_2(S,T)$, but $L \tensor_{\FF_2(S,T)} R$ remains reduced and quasi-$F$-split. \\

This suggests that more careful approaches when base changing to an $F$-finite field may be useful. In \cite{HochsterHuneke94}, Hochster and Huneke developed machinery to reduce questions about complete local rings to questions about $F$-finite rings called the $\Gamma$-construction. Though initially constructed to show the existence of test elements in essentially of finite type algebras over excellent local rings, the $\Gamma$-construction has found wide use in extending results on $F$-finite rings to the complete, local, and non-$F$-finite setting. See \cite{EnescuHochsterFrobeniusstructureoflocalcohomology}, \cite{Ma_FinitenessPropertiesofLocalCohomologyforFPureLocalRings}, and \cite{Lyu_TheGammaConstructionAndPermanencePropertiesofTheRelativeFRationalSignature} for more examples of such extensions. \\

For a complete local ring $\locr$, one can take a purely inseparable extension $\fK^\Gamma/\fK$ which adjoins, for all $e > 0$, all $p^e$th-roots for all but finitely elements $\Gamma$ of the $p$-base of $\fK$. Remarkably $R^\Gamma := \fK^{\Gamma} \ctensor_\fK R$ is not only $F$-finite, but preserves nearly every singularity type of $R$ if $\Gamma$ is chosen sufficiently small, but still cofinite in the $p$-base $\Lambda$ (\cite[Theorem 3.4]{Murayama_TheGammaConstructionAndAsymptoticInvariantsofLineBundlesoverArbitraryFields} provides an overview for such results). We conjecture that the $\Gamma$-construction is also capable of preserving quasi-$F$-pure/quasi-$F$-split height.
\begin{conjecture*}
    Let $\locr$ be a complete local Noetherian ring of characteristic $p$. For sufficiently small $\Gamma \subset \Lambda$, $\height(R) = \height(R^\Gamma)$. 
\end{conjecture*}
If true, this would provide a systematic way to construct an $F$-finite reduction of any quasi-$F$-split ring such that their height, along with most other properties, are preserved. 

\bibliographystyle{amsalpha}
\bibliography{references}
\end{document}